\documentclass[english]{article}
\usepackage[T1]{fontenc}
\usepackage[latin9]{inputenc}
\usepackage{amsthm}
\usepackage{amsmath}
\usepackage{amssymb}
\usepackage{graphicx}
\usepackage{esint}

\makeatletter
\numberwithin{equation}{section}
  \theoremstyle{plain}
  \newtheorem{lem}{\protect\lemmaname}[section]
  \theoremstyle{plain}
  \newtheorem{thm}{\protect\theoremname}[section]
  \theoremstyle{remark}
  \newtheorem{rem}{\protect\remarkname}[section]
  \theoremstyle{definition}
  \newtheorem{defn}{\protect\definitionname}[section]
  \theoremstyle{plain}
  \newtheorem{prop}{\protect\propositionname}[section]

\makeatother

\usepackage{babel}
  \providecommand{\definitionname}{Definition}
  \providecommand{\lemmaname}{Lemma}
  \providecommand{\propositionname}{Proposition}
  \providecommand{\remarkname}{Remark}
\providecommand{\theoremname}{Theorem}

\begin{document}

\title{Simple Piecewise Geodesic Interpolation of Simple and Jordan Curves
with Applications}

\author{H. Boedihardjo%
\thanks{The Oxford-Man Institute, Eagle House, Walton Well Road, Oxford OX2 6ED.
\protect \\
Email: horatio.boedihardjo@oxford-man.ox.ac.uk %
} $\ $and X. Geng%
\thanks{Mathematical Institute, University of Oxford, Oxford OX2 6GG and the
Oxford-Man Institute, Eagle House, Walton Well Road, Oxford OX2 6ED.\protect \\
Email: xi.geng@maths.ox.ac.uk %
}}

\maketitle

\begin{abstract}
We explicitly construct simple, piecewise minimizing geodesic, arbitrarily
fine interpolation of simple and Jordan curves on a Riemannian manifold.
In particular, a finite sequence of partition points can be specified
in advance to be included in our construction. Then we present two
applications of our main results: the generalized Green's theorem
and the uniqueness of signature for planar Jordan curves with finite
$p$-variation for $1\leqslant p<2.$\\
\\
\textbf{Keywords }$\ $Piecewise Geodesic Interpolation $\cdot$  Simple Curves $\cdot$ Jordan Curves $\cdot$ Generalized
Green's Theorem $\cdot$ Uniqueness of Signature\\
\textbf{Mathematics Subject Classification (2000) }26A42 $\cdot$\textbf{
}41A05 $\cdot$ 53C22
\end{abstract}

\section{Introduction }

The classical proofs of many properties of Jordan curves (e.g. the Jordan curve theorem) or functions on Jordan curves (e.g. Cauchy's theorem) begin with the consideration of polygonal Jordan curves. As part of the proof of the Jordan curve theorem in \cite{Proof of Jordan curve theroem},
it was shown that for every planar Jordan curve, there is a polygonal Jordan curve that approximates the original Jordan curve arbitrarily well. Here we shall prove a stronger and more general fact that given a Jordan curve on a connected Riemannian manifold M and $n$ points on the curve, there exists a simple, piecewise minimizing geodesic, arbitrarily fine interpolation which contains these $n$ points as interpolation points. The proof relies on another main result of this paper for non-closed simple curves. Such case was first treated by Werness \cite{Wer12}, in which the author used an inductive but non-constructive method. Here we provide another proof of this result, which has the advantage of being explicit and constructive, and hence numerically computable.

We would like to emphasize that our approximation, unlike that of
\cite{Proof of Jordan curve theroem} which is a direct consequence of our result, does not rely on the flatness of the Euclidean metric, and respects the parametrization of the curve, i.e. it is an interpolation rather than merely an approximation in the uniform norm. The latter is particularly important for applications in the context of rough path theory, where we approximate continuous paths by bounded variation ones in the $p$-variation metric. Such idea is fundamental to the study of the roughness of continuous paths, and particularly of sample paths of continuous stochastic processes (see \cite{fBM}, \cite{friz}).

We also give two applications of our main result.

Taking advantage of the fact that the $p$-variation of a piecewise
linear interpolation of a path is bounded by the $p$-variation of
the path itself, our approximation theorem gives immediately Green's
theorem for planar Jordan curves with finite $p$-variation, where
$1\leqslant p<2$. To our best knowledge, in the rough path literature,
the only other attempt in extending Green's theorem to
non-rectifiable curves appeared in \cite{Yamthesis}, where Green's
theorem was proved for the boundaries of $\alpha$-Hölder domains for $\frac{1}{3}<\alpha<1$. Our result is a partial generalization
of Yam's. Yam's result requires the curve to be $\alpha$-Hölder continuous
under the conformal parametrization whereas our result only requires
the curve to be $\alpha$-Hölder continuous under some parametrization. 

A long-standing open problem in rough path theory is to what extent
a path can be determined from its iterated integrals of any order. This is
usually known as the uniqueness of signature problem. Hambly and Lyons
\cite{tree like} proved that the iterated integrals of a rectifiable
curve vanish if and only if the path is tree-like, based on a similar
type of approximation result for tree-like paths. Using our approximation
result, we prove the uniqueness of signature for planar Jordan curves
with finite $p$-variations, where $1\leqslant p<2$. The case of
non-closed simple curves was treated in \cite{simple curve uniqueness }.
To our best knowledge this is the strongest uniqueness of signature
result so far for non-rectifiable curves. 

Throughout the rest of this paper, all curves are assumed to be continuous.

\section{Simple Piecewise Geodesic Interpolation of Simple and Jordan Curves}

In this section, we are going to prove our main results about simple
piecewise geodesic approximation of simple and Jordan curves in Riemannian
manifolds. Although the most interesting and nontrivial case lies
in the Euclidean plane, we formulate our problems in a Riemannian
geometric setting of arbitrary dimension since the proofs do not rely
on Euclidean geometry (that is, the ``flatness'' of Euclidean metric)
at all. 

Throughout this section, let $M$ be a $d$-dimensional connected
Riemannian manifold ($d\geqslant2$).

The following lemma, which is an easy fact from Riemannian geometry,
is fundamental for us to formulate our main results. 
\begin{lem}
\label{minimizing geodesic}For any compact set $K\subset M$, there
exists some $\varepsilon=\varepsilon_{K}>0,$ such that for any $x,y\in K$
with 
\[
d\left(x,y\right)<\varepsilon,
\]
 there exists a unique minimizing geodesic in $M$ joining $x$ and
$y$, where $d\left(\cdot,\cdot\right)$ denotes the Riemannian distance
function.\end{lem}
\begin{proof}
For any $x\in K,$ choose $\delta_{x}$ small enough such that $B\left(x,\delta_{x}\right)$
is a geodesically convex normal ball (see \cite{do carmo}, p. 76,
Proposition 4.2). By compactness, we have a finite covering of $K:$
\[
K\subset\bigcup_{i=1}^{k}B\left(x_{i},\frac{\delta_{x_{i}}}{2}\right),
\]
where $x_{1},\cdots,x_{k}\in K.$ Let $\varepsilon=\frac{1}{2}\min\left\{ \delta_{x_{1}},\cdots,\delta_{x_{k}}\right\} .$
It follows that for any $x,y\in K$ with $d\left(x,y\right)<\varepsilon,$
there exists some $1\leqslant i\leqslant k,$ such that $x,y\in B\left(x_{i},\delta_{x_{i}}\right)$.
Therefore, by geodesic convexity we know that $x$ and $y$ can be
joined by a unique minimizing geodesic in $M$ which lies in $B\left(x_{i},\delta_{x_{i}}\right).$ 
\end{proof}

Now we are in position to state our main results.

The first main result is a simple piecewise geodesic approximation
theorem for non-closed simple curves in $M$.
\begin{thm}
\label{simple curve approximation} Let $\gamma$ be a  non-closed
simple curve in $M$. For all $\varepsilon>0,$ there exists a finite
partition 
\[
\mathcal{P}_{\varepsilon}:\ 0=t_{0}<t_{1}<\cdots<t_{n-1}<t_{n}=1
\]
of $[0,1],$ such that 

(1) the mesh size of the partition $\|\mathcal{P}_{\varepsilon}\|=\max_{i=1,\cdots,n}\left(t_{i}-t_{i-1}\right)<\varepsilon$;

(2) for any $i=1,\cdots,n,$ $\gamma_{t_{i-1}}$ and $\gamma_{t_{i}}$
can be joined by a unique minimizing geodesic in $M$, and the piecewise
geodesic interpolation (more precisely, piecewise minimizing geodesic
interpolation, and the same thereafter) $\gamma^{\mathcal{P}_{\varepsilon}}$
of $\gamma$ over the partition points in $\mathcal{P}_{\varepsilon}$
is a simple curve.
\end{thm}

The proof of Theorem \ref{simple curve approximation} relies on the
following crucial lemma, which depends heavily on properties of minimizing
geodesics. In the Euclidean case, we illustrate the lemma in Figure
1, which says that if the lengths of the straight line segments $\overline{xy}$
and $\overline{zw}$ are both less than or equal to $r$, then at
least one of the four line segments $\overline{xz},\overline{xw},\overline{yz},\overline{yw}$
has length strictly less than $r$. 

\begin{figure}
\begin{center}

\includegraphics[scale=0.28]{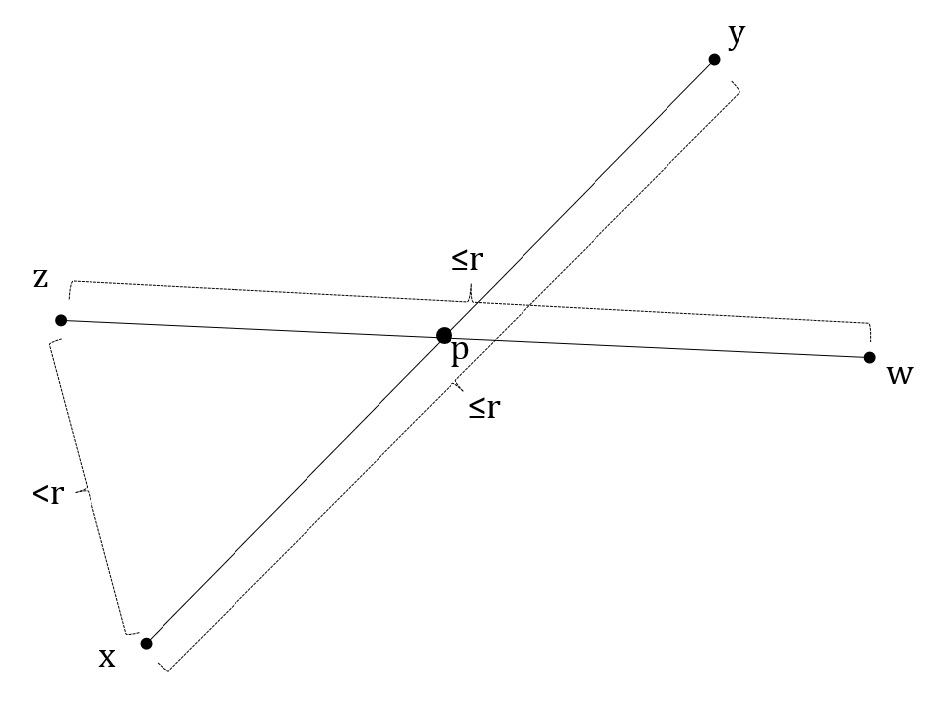}
\caption{This figure illustrates the relative positions of the points in Lemma
\ref{geodesic fact 1} in the Euclidean case. The lengths of the line
segments $\overline{xy}$ and $\overline{zw}$ are less than or equal
to $r$. Here the length of $\overline{zx}$ is strictly less than
$r$. }

\end{center}
\end{figure}

\begin{lem}
\label{geodesic fact 1} Let $x,y,z,w\in M$ and $\alpha:\ \left[0,1\right]\rightarrow M$
(respectively, $\beta:\ \left[0,1\right]\rightarrow M$) be a minimizing
geodesic joining $x$ and $y$ (respectively, $z$ and $w$). Assume
that $\alpha\left(\left[0,1\right]\right)\bigcap\beta\left(\left[0,1\right]\right)\neq\emptyset$
and for some $r>0$, $d\left(x,y\right)\leqslant r,\ d\left(z,w\right)\leqslant r$.
Then at least one of $d\left(x,z\right),d\left(y,z\right),d\left(x,w\right),d\left(y,w\right)$
is strictly less than $r$.\end{lem}
\begin{proof}
Without loss of generality, we shall assume that all geodesics are parametrized
at constant speed. Let $\alpha\left(u\right)=\beta\left(v\right)=p$
for some $u,v\in\left[0,1\right].$ Since $\alpha$ and $\beta$ are
minimizing geodesics, we know that
\begin{eqnarray*}
d\left(x,y\right)=d\left(x,p\right)+d\left(p,y\right) & \leqslant & r,\\
d\left(z,w\right)=d\left(z,p\right)+d\left(p,w\right) & \leqslant & r.
\end{eqnarray*}
Therefore, at least one of the following four cases happens:

(1) $d\left(x,p\right)\leqslant\frac{r}{2},\ d\left(z,p\right)\leqslant\frac{r}{2}$;

(2) $d\left(x,p\right)\leqslant\frac{r}{2},\ d\left(p,w\right)\leqslant\frac{r}{2}$;

(3) $d\left(p,y\right)\leqslant\frac{r}{2},\ d\left(z,p\right)\leqslant\frac{r}{2}$;

(4) $d\left(p,y\right)\leqslant\frac{r}{2};\ d\left(p,w\right)\leqslant\frac{r}{2}.$

First assume that (1) holds. It follows that
\[
d\left(x,z\right)\leqslant d\left(x,p\right)+d\left(z,p\right)\leqslant r.
\]
If $d\left(x,z\right)=r,$ then
\[
d(x,p)=d(z,p)=\frac{r}{2},
\]
and hence (4) holds, which implies 
\[
d\left(y,w\right)\leqslant d\left(p,y\right)+d\left(p,w\right)\leqslant r.
\]
If $d\left(y,w\right)=r,$ then 
\[
d\left(p,y\right)=d\left(p,w\right)=\frac{r}{2}.
\]
Consequently, we have $u=v=\frac{1}{2}.$

Now define 
\[
\widetilde{\alpha}\left(t\right)=\begin{cases}
\alpha\left(t\right), & t\in\left[0,\frac{1}{2}\right];\\
\beta\left(1-t\right), & t\in\left[\frac{1}{2},1\right].
\end{cases}
\]
Since 
\[
\mbox{Length}\left(\widetilde{\alpha}\right)=r=d\left(x,z\right),
\]
$\widetilde{\alpha}$ is minimizing. Moreover, since any geodesic
has constant speed, by definition we know that $\widetilde{\alpha}$
is parametrized proportionally to arc length. It follows from the
first variation formula (see \cite{do carmo}, p. 195, Proposition
2.4) that $\widetilde{\alpha}$ must be a geodesic. However, since
$\widetilde{\alpha}|_{\left[0,\frac{1}{2}\right]}=\alpha|_{\left[0,\frac{1}{2}\right]},$
by the uniqueness of geodesics we have $\widetilde{\alpha}=\alpha$
and hence $y=z.$ Similarly we have $x=w.$

The other cases can be treated in the same way, which completes the
proof of the lemma.
\end{proof}

With the help of Lemma \ref{geodesic fact 1}, we can now prove Theorem
\ref{simple curve approximation}. The key idea is to construct a
sequence of times $t_{1},t_{2},\ldots$, such that $t_{i+1}$ is the
\textit{last} exit time of $\gamma$ from a small geodesic ball around
$\gamma_{t_{i}}$ after time $t_{i}$. The uniform continuity of the
inverse of the map $t\rightarrow\gamma_{t}$ will guarantee that $t_{i}$
and $t_{i+1}$ are close. We then need to argue that adjacent geodesic
segments as well as non-adjacent geodesic segments in the approximation
curve do not intersect. The latter uses Lemma \ref{geodesic fact 1}.
We illustrate the first step of the construction in Figure 2.

\begin{proof}[Proof of Theorem 2.1] Fix $\varepsilon>0.$ Since $\gamma$
is a continuous and injective mapping from the compact space $\left[0,1\right]$
to the Hausdorff space $M,$ it is a homeomorphism from $\left[0,1\right]$
to its image. By compactness and hence uniform continuity of $\gamma^{-1}$
we know that there exists $\delta_{\varepsilon}>0$ such that for
any $s,t\in\left[0,1\right],$
\[
d\left(\gamma_{s},\gamma_{t}\right)<\delta_{\varepsilon}\implies\left|t-s\right|<\varepsilon.
\]
We further assume that $\delta_{\varepsilon}<\varepsilon_{\gamma\left(\left[0,1\right]\right)}$,
where $\varepsilon_{\gamma\left(\left[0,1\right]\right)}$ is the
positive number in Lemma \ref{minimizing geodesic} depending on the
compact set $\gamma\left(\left[0,1\right]\right)\subset M.$ It follows
from Lemma \ref{minimizing geodesic} that for any $s,t\in\left[0,1\right]$
with $d\left(\gamma_{s},\gamma_{t}\right)<\delta_{\varepsilon},$
$\gamma_{s}$ and $\gamma_{t}$ can be joined by a unique minimizing
geodesic in $M$. Now define an increasing sequence of points $\left\{ t_{i}\right\} {}_{i=0}^{\infty}$
in $\left[0,1\right]$ inductively by setting $t_{0}=0$ and
\[
t_{i}=\sup\left\{ t\in\left[t_{i-1},1\right]:\ \gamma_{t}\in\overline{B}\left(\gamma_{t_{i-1}},\frac{\delta_{\varepsilon}}{2}\right)\right\} ,\ i\geqslant1.
\]

\begin{figure}
\begin{center}

\includegraphics[scale=0.34]{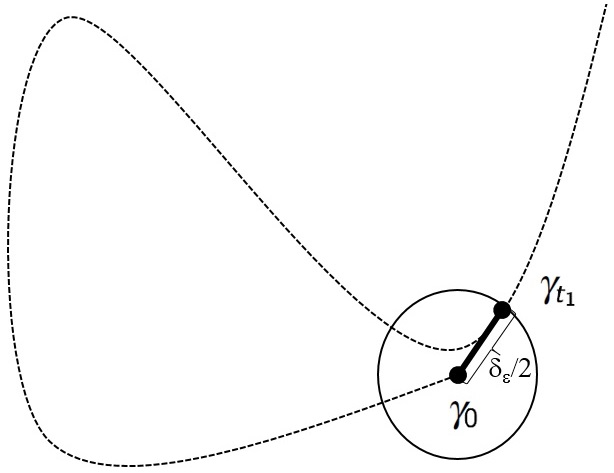}
\caption{This figure illustrates the first step in the construction given in
the proof of Theorem \ref{simple curve approximation}. The dotted
line represents the simple curve $\gamma$. The solid geodesic segment
joining the point $\gamma_{0}$ and $\gamma_{t_{1}}$ represents the
first step in the construction of the piecewise geodesic interpolation
of $\gamma$. Note that we take $t_{1}$ to be the last exit time
of $\gamma$ from a $\frac{\delta_{\varepsilon}}{2}$-geodesic ball
centered at $\gamma_{0}$.}

\end{center}
\end{figure}

We claim that there exists some $l\geqslant1,$ such that for all
$i\geqslant l$, $t_{i}=1$. In fact, if it is not the case, then
for any $i\geqslant1,$ we have 
\[
t_{i-1}<t_{i}<1\ \mbox{and}\ d\left(\gamma_{t_{i-1}},\gamma_{t_{i}}\right)=\frac{\delta_{\varepsilon}}{2}.
\]
On the other hand, by the uniform continuity of $\gamma,$ there exists
some $\eta_{\varepsilon}>0,$ such that for any $s,t\in\left[0,1\right],$
\[
\left|t-s\right|<\eta_{\varepsilon}\implies d\left(\gamma_{s},\gamma_{t}\right)<\frac{\delta_{\varepsilon}}{2}.
\]
Therefore, for any $i\geqslant1,$ $\left|t_{i}-t_{i-1}\right|\geqslant\eta_{\varepsilon},$
which is an obvious contradiction. Now set 
\[
l=\min\left\{ i\geqslant1:\ t_{i}=1\right\} ,
\]
and define
\[
\mathcal{P}_{\varepsilon}:\ 0=t_{0}<t_{1}<\cdots<t_{l-1}<t_{l}=1
\]
to be a finite partition of $\left[0,1\right].$ Then it is easy to
see that $\left\Vert \mathcal{P}_{\varepsilon}\right\Vert <\varepsilon$,
where $\left\Vert \mathcal{P}_{\varepsilon}\right\Vert $ denote the
mesh size of the partition $\mathcal{P}_{\varepsilon}$.

It remains to show that the piecewise geodesic interpolation $\gamma^{\mathcal{P}_{\varepsilon}}$
of $\gamma$ over the points of $\mathcal{P}_{\varepsilon}$ is a
simple curve. 

To see this, first notice that for adjacent intervals $\left[t_{i-1},t_{i}\right],\ \left[t_{i},t_{i+1}\right],$
we have 
\[
\gamma^{\mathcal{P}_{\varepsilon}}|_{\left[t_{i-1},t_{i}\right]}\bigcap\gamma^{\mathcal{P}_{\varepsilon}}|_{\left[t_{i},t_{i+1}\right]}=\left\{ \gamma_{t_{i}}\right\} .
\]
In fact, if it is not the case, then there exist $s_{1}\in\left[t_{i-1},t_{i}\right)$
and $s_{2}\in\left(t_{i},t_{i+1}\right]$ such that 
\[
\gamma_{s_{1}}^{\mathcal{P}_{\varepsilon}}=\gamma_{s_{2}}^{\mathcal{P}_{\varepsilon}}\neq\gamma_{t_{i}}.
\]

If $i<l-1$, then by applying Lemma \ref{minimizing geodesic} with
$x=\gamma_{t_{i}}$ and $y=\gamma_{s_{1}}^{\mathcal{P}_{\varepsilon}}=\gamma_{s_{2}}^{\mathcal{P}_{\varepsilon}}$,
$\gamma^{\mathcal{P}_{\varepsilon}}|_{\left[s_{1},t_{i}\right]}$
is a reparametrization of the reversal of $\gamma^{\mathcal{P}_{\varepsilon}}|_{\left[t_{i},s_{2}\right]}$,
which we denote as $\overleftarrow{\gamma^{\mathcal{P}_{\varepsilon}}}|_{\left[t_{i},s_{2}\right]}$.
In particular, $\gamma^{\mathcal{P}_{\varepsilon}}|_{\left[t_{i},t_{i+1}\right]}$
and $\overleftarrow{\gamma^{\mathcal{P}_{\varepsilon}}}|_{\left[t_{i-1},t_{i}\right]}$
are geodesics that starts at the same position with the same initial
velocity. By the uniqueness of geodesic, either $\gamma^{\mathcal{P}_{\varepsilon}}\left(\left[t_{i},t_{i+1}\right]\right)\subseteq\overleftarrow{\gamma^{\mathcal{P}_{\varepsilon}}}\left(\left[t_{i-1},t_{i}\right]\right)$
or $\overleftarrow{\gamma^{\mathcal{P}_{\varepsilon}}}\left(\left[t_{i-1},t_{i}\right]\right)\subseteq\gamma^{\mathcal{P}_{\varepsilon}}\left(\left[t_{i},t_{i+1}\right]\right)$.
In particular we have either $\gamma^{\mathcal{P}_{\varepsilon}}|_{\left[t_{i-1},t_{i}\right]}$
passes through $\gamma_{t_{i+1}}$ or $\gamma^{\mathcal{P}_{\varepsilon}}|_{\left[t_{i},t_{i+1}\right]}$
passes through $\gamma_{t_{i-1}}$. As $\overleftarrow{\gamma^{\mathcal{P}_{\varepsilon}}}|_{\left[t_{i-1},t_{i}\right]}$
and $\gamma^{\mathcal{P}_{\varepsilon}}|_{\left[t_{i},t_{i+1}\right]}$
are minimizing geodesics and we have $d\left(\gamma_{t_{i}},\gamma_{t_{i-1}}\right)=d\left(\gamma_{t_{i}},\gamma_{t_{i+1}}\right)$,
we conclude that $\gamma_{t_{i-1}}=\gamma_{t_{i+1}}$ which contradicts
that $\gamma$ is simple. Figure 3 illustrates this argument. 

\begin{figure}
\begin{center}

\includegraphics[scale=0.2]{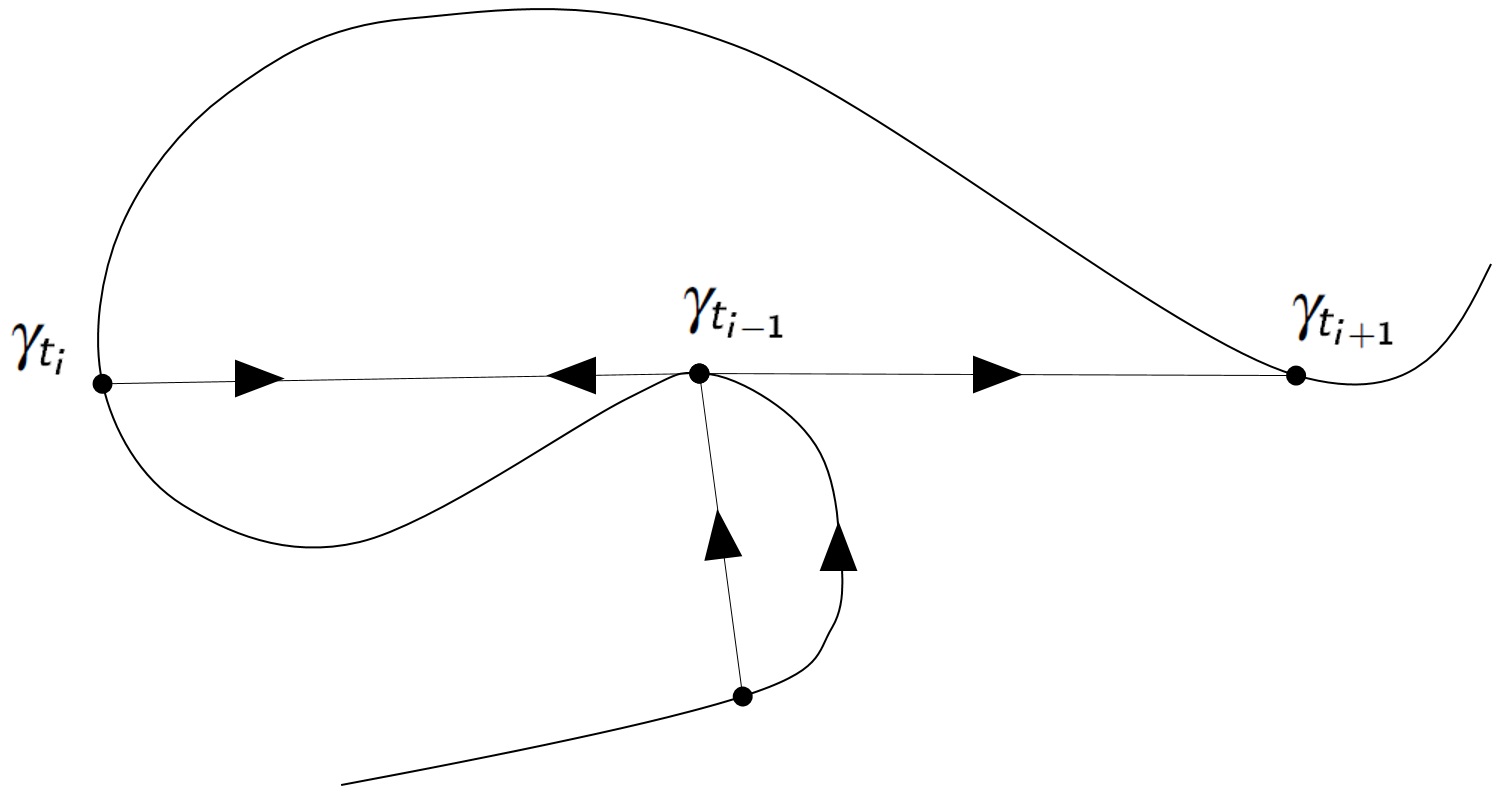}
\caption{This figure illustrates the argument in the proof of Theorem \ref{simple curve approximation}
that two adjacent line segment of the approximation curve we constructed
cannot intersect. The straight line represents the geodesic segments
in the piecewise geodesic interpolation of $\gamma$. $\gamma_{t_{i-1}},\gamma_{t_{i}},\gamma_{t_{i+1}}$
are subdivision points of the curve. If the two adjacent line segments
do intersect, as in the figure below, then $\gamma_{t_{i+1}}$ would
be closer to $\gamma_{t_{i-1}}$ than to $\gamma_{t_{i}}$ which would
contradict our construction. }

\end{center}
\end{figure}

If $i=l-1$, then arguing as in the case $i<l-1$, we have either
$\gamma^{\mathcal{P}_{\varepsilon}}\left(\left[t_{i},t_{i+1}\right]\right)\subseteq\overleftarrow{\gamma^{\mathcal{P}_{\varepsilon}}}\left(\left[t_{i-1},t_{i}\right]\right)$
or $\overleftarrow{\gamma^{\mathcal{P}_{\varepsilon}}}\left(\left[t_{i-1},t_{i}\right]\right)\subseteq\gamma^{\mathcal{P}_{\varepsilon}}\left(\left[t_{i},t_{i+1}\right]\right)$.
However, as $i=l-1$, we have $d\left(\gamma_{t_{i}},\gamma_{t_{i+1}}\right)\leqslant\frac{\delta_{\varepsilon}}{2}=d\left(\gamma_{t_{i-1}},\gamma_{t_{i}}\right)$
and hence $\gamma^{\mathcal{P}_{\varepsilon}}\left(\left[t_{i},t_{i+1}\right]\right)\subseteq\overleftarrow{\gamma^{\mathcal{P}_{\varepsilon}}}\left(\left[t_{i-1},t_{i}\right]\right)$.
In particular, $\gamma^{\mathcal{P}_{\varepsilon}}|_{\left[t_{i-1},t_{i}\right]}$
passes through $\gamma_{t_{i+1}}$. Therefore, $d\left(\gamma_{t_{i-1}},\gamma_{t_{i+1}}\right)\leqslant d\left(\gamma_{t_{i-1}},\gamma_{t_{i}}\right)$
which contradicts the construction of $\{t_{i}\}_{i=0}^{l}.$ 

On the other hand, if $\left[t_{i-1},t_{i}\right]$ and $\left[t_{j-1},t_{j}\right]$
($i<j$) are non-adjacent intervals and 
\[
\gamma^{\mathcal{P}_{\varepsilon}}|_{\left[t_{i-1},t_{i}\right]}\bigcap\gamma^{\mathcal{P}_{\varepsilon}}|_{\left[t_{j-1},t_{j}\right]}\neq\emptyset,
\]
then by Lemma \ref{geodesic fact 1} we know that at least one of
\[
d\left(\gamma_{t_{i-1}},\gamma_{t_{j-1}}\right),d\left(\gamma_{t_{i}},\gamma_{t_{j-1}}\right),d\left(\gamma_{t_{i-1}},\gamma_{t_{j}}\right),d\left(\gamma_{t_{i}},\gamma_{t_{j}}\right)
\]
is strictly less than $\frac{\delta_{\varepsilon}}{2}$. However,
this again contradicts the construction of $\left\{ t_{i}\right\} {}_{i=0}^{l}.$

Now the proof is complete.\end{proof}

The same technique of proof will allow us to prove our second main
result, which is concerned with simple piecewise geodesic approximations
of Jordan curves. This result significantly strengthens Theorem \ref{simple curve approximation}.
\begin{thm}
\label{Jordan curve approximation} Let $\gamma:\ \left[0,1\right]\rightarrow M$
be a Jordan curve. Assume that $0<\tau_{1}<\cdots<\tau_{k}<1$ are
$k$ fixed points in $\left[0,1\right].$ Then for any $\varepsilon>0,$
there exists a finite partition 
\[
\mathcal{P}_{\varepsilon}:\ 0=t_{0}<t_{1}<\cdots<t_{n-1}<t_{n}=1
\]
of $\left[0,1\right]$, such that 

(1) $\tau_{1},\cdots,\tau_{k}$ are partition points of $\mathcal{P}_{\varepsilon}$;

(2) $\left\Vert \mathcal{P}_{\varepsilon}\right\Vert <\varepsilon$;

(3) for $i=1,\cdots,n,$ $\gamma_{t_{i-1}}$ and $\gamma_{t_{i}}$
can be joined by a unique minimizing geodesic in $M$, and the piecewise
geodesic interpolation $\gamma^{\mathcal{P}_{\varepsilon}}$ of $\gamma$
over the partition points in $\mathcal{P}_{\varepsilon}$ is a  Jordan
curve.
\end{thm}

The proof of Theorem \ref{Jordan curve approximation} relies on the
following geometric fact. It is illustrated by Figure 4.

\begin{figure}
\begin{center}

\includegraphics[scale=0.22]{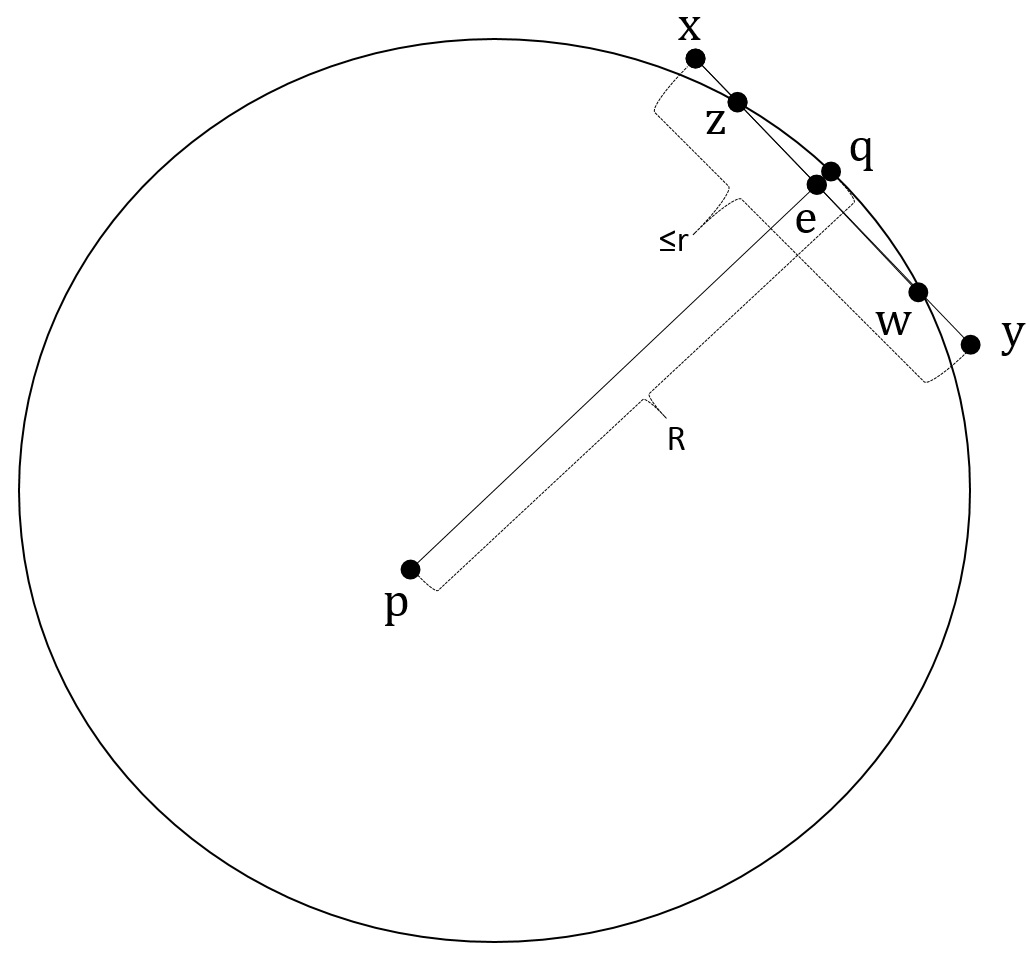}
\caption{This figure illustrates the relative positions of the points involved
in Lemma \ref{geodesic fact 2}.}

\end{center}
\end{figure}

\begin{lem}
\label{geodesic fact 2} Let $B\left(p,R\right)$ be a geodesically
convex normal ball centered at $p\in M$, and let $q\in\partial B\left(p,R\right).$
Assume that $x,y\in\overline{B}\left(p,R\right){}^{c}$ and there
exists a minimizing geodesic $\alpha:\ \left[0,1\right]\rightarrow M$
joining $x$ and $y.$ If $\alpha\left(\left[0,1\right]\right)\bigcap\overline{pq}\neq\emptyset$
and $d\left(x,y\right)\leqslant r$ for some $0<r<R,$ where $\overline{pq}$
denotes the image of the unique minimizing geodesic in $M$ joining
$p$ and $q$, then
\[
d\left(x,q\right)<r,\ d\left(y,q\right)<r.
\]
\end{lem}
\begin{proof}
The conclusion is obvious if $q\in\alpha\left(\left[0,1\right]\right).$
Otherwise, let $t\in\left(0,1\right)$ be the unique time such that
$e:=\alpha\left(t\right)\in B\left(p,R\right)$ is the intersection
point of $\alpha\left(\left[0,1\right]\right)$ and $\overline{pq}$.
By using the fact that $B\left(p,R\right)$ is a geodesically convex
normal ball, it is easy to show that there exists a unique $u\in\left(0,t\right)$
and a unique $v\in\left(t,1\right),$ such that $z:=\alpha\left(u\right)$
and $w:=\alpha\left(v\right)$ lie on $\partial B\left(p,R\right).$
Observe that $e$ and $q$ are distinct, for their equality would
contradict the fact that $r<R.$ Now it follows from properties of
minimizing geodesics that
\begin{eqnarray*}
d\left(x,q\right) & \leqslant & d\left(x,e\right)+d\left(e,q\right)\\
 & = & d\left(x,e\right)+d\left(p,q\right)-d\left(p,e\right)\\
 & = & d\left(x,e\right)+d\left(p,w\right)-d\left(p,e\right)\\
 & \leqslant & d\left(x,e\right)+d\left(e,w\right)\\
 & = & d\left(x,w\right)\\
 & < & d\left(x,y\right)\\
 & \leqslant & r.
\end{eqnarray*}
Similarly, we have $d\left(y,q\right)<r.$
\end{proof}

Now we can prove Theorem \ref{Jordan curve approximation}. Our proof
is constructive and the idea is as follows. Recall that the times
$\tau_{1},\ldots,\tau_{k}$ should to included in our partition. Firstly,
We find small disjoint geodesic balls around the points $\gamma_{\tau_{i}},\ldots,\gamma_{\tau_{k}},\gamma_{1}$.
Secondly, we connect each point $\gamma_{\tau_{i}}$ by two radial
minimizing geodesics to the point where $\gamma$ \textit{first} enters
the geodesic ball around $\gamma_{\tau_{i}}$ before time $\tau_{i}$
and to the point where $\gamma$ \textit{last} exists the geodesic
ball. Finally, we construct simple piecewise geodesic interpolation
for each piece of simple curves outside those geodesic balls inductively,
by using the algorithm in Theorem \ref{simple curve approximation}.
To make sure that those approximation curves do not intersect the
geodesic segments inside those geodesic balls, we need to use Lemma
\ref{geodesic fact 2}. Figure 5 illustrates the idea when $k=2.$

\begin{proof}[Proof of Theorem 2.2]Take an arbitrary $\tau\in\left(0,\tau_{1}\right).$
Since $\gamma$ is a Jordan curve, we know that $\gamma_{\tau},\gamma_{\tau_{1}},\cdots,\gamma_{\tau_{k}},\gamma_{\tau_{k+1}}\in M$
are all distinct, where we set $\tau_{k+1}=1$. By the Hausdorff property,
there exists some $\delta>0$ such that the closed metric balls $\overline{B}\left(\gamma_{\tau_{1}},\delta\right),\cdots,\overline{B}\left(\gamma_{\tau_{k+1}},\delta\right)$
are all disjoint and $\gamma_{\tau}\notin\bigcup_{i=1}^{k+1}\overline{B}\left(\gamma_{\tau_{i}},\delta\right)$. 

For the moment, by periodic extension and restriction we regard $\gamma$
as defined on $\left[\tau,\tau+1\right]$ with starting and end points
being $\gamma\left(\tau\right).$ 

Now fix $\varepsilon>0.$ Without loss of generality we assume that
\[
\varepsilon<\mbox{min}\left\{ \tau,\tau_{1}-\tau,\frac{\tau_{2}-\tau_{1}}{2},\cdots,\frac{\tau_{k+1}-\tau_{k}}{2}\right\} .
\]

First of all, by the uniform continuity of $\gamma|_{\left[\tau,\tau_{i}\right]}^{-1}$
and $\gamma|_{\left[\tau_{i},\tau+1\right]}^{-1}$, there exists some
$\delta_{\varepsilon}>0,$ such that for all $i=1,\cdots,k+1,$ any
$s,t\in\left[\tau,\tau_{i}\right]$ or $s,t\in\left[\tau_{i},\tau+1\right],$
\[
d\left(\gamma_{s},\gamma_{t}\right)<\delta_{\varepsilon}\implies\left|t-s\right|<\varepsilon.
\]
Now set $U_{i}=B\left(\gamma_{\tau_{i}},\delta_{\varepsilon}\right)$.
Here we assume that $\delta_{\varepsilon}$ is small enough so that
each $U_{i}$ is a geodesically convex normal ball and Lemma \ref{minimizing geodesic}
holds for those $\gamma_{s},\gamma_{t}$ with $d\left(\gamma_{s},\gamma_{t}\right)<2\delta_{\varepsilon}.$
Define 
\begin{eqnarray*}
u_{i} & = & \inf\left\{ t\in\left[\tau,\tau_{i}\right]:\ \gamma_{t}\in\overline{U_{i}}\right\} ,\\
v_{i} & = & \sup\left\{ t\in\left[\tau_{i},\tau+1\right]:\ \gamma_{t}\in\overline{U_{i}}\right\} .
\end{eqnarray*}

\begin{figure}
\begin{center}

\includegraphics[scale=0.2]{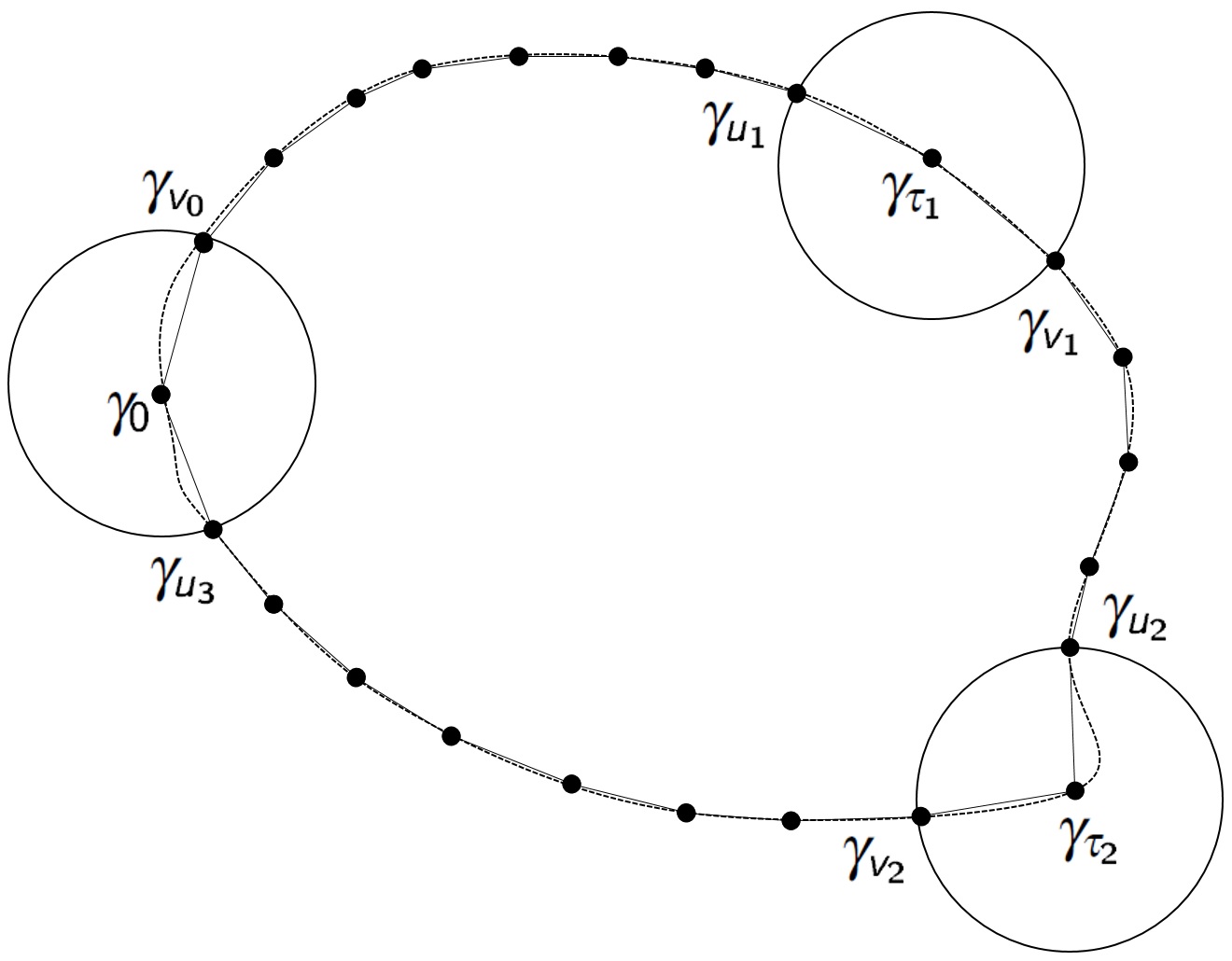}
\caption{This figure illustrates the idea of proving of Theorem 2.2 when $k=2.$
The dotted line represents the curve $\gamma$. The solid line represents
the piecewise geodesic interpolation of $\gamma$.}

\end{center}
\end{figure}

To return to the original time interval $\left[0,1\right],$ let $v_{0}=v_{k+1}-1.$
We have $|v_{0}|<\varepsilon$, $\left|\tau_{i}-u_{i}\right|<\varepsilon$,
$\left|v_{i}-\tau_{i}\right|<\varepsilon$ and
\[
0<v_{0}<u_{1}<\tau_{1}<v_{1}<\cdots<u_{k}<\tau_{k}<v_{k}<u_{k+1}<1,
\]

and 
\[
\gamma_{u_{i}}\neq\gamma_{v_{i}},\ d\left(\gamma_{\tau_{i}},\gamma_{u_{i}}\right)=d\left(\gamma_{\tau_{i}},\gamma_{v_{i}}\right)=\delta_{\varepsilon}.
\]
Moreover, we have 
\[
\gamma|_{\left(v_{0},u_{1}\right)\cup\left(v_{1},u_{2}\right)\cup\cdots\cup\left(v_{k-1},u_{k}\right)\cup\left(v_{k},u_{k+1}\right)}\bigcap\left(\bigcup_{i=1}^{k+1}\overline{U_{i}}\right)=\emptyset.
\]

We will take $v_{0},u_{1},\tau_{1},v_{1},\cdots,u_{k},\tau_{k},v_{k},u_{k+1}$
as part of the partition points in $\mathcal{P}_{\varepsilon}$. In
particular, $v_{0}$ will be the first point, $u_{k+1}$ will be the
last point (except $0$ and $1$), and $u_{i},\tau_{i},v_{i}$ will
be successive points in $\mathcal{P}_{\varepsilon},$ so the piecewise
geodesic interpolation of $\gamma$ over those small intervals is
a finite sequence of radial geodesics of the balls centered at $\gamma_{\tau_{i}}$
with radius $\delta_{\varepsilon}$ for $i=1,\cdots,k+1$.

For the next step, notice that $\gamma|_{\left[v_{0},u_{1}\right]},\gamma|_{\left[v_{1},u_{2}\right]},\cdots,\gamma|_{\left[v_{k},u_{k+1}\right]}$
are $k+1$ non-closed simple curves with disjoint images. We are going
to use the constructive procedure in the proof of Theorem \ref{simple curve approximation}
to define a simple piecewise geodesic approximation of each $\gamma|_{\left[v_{i-1},u_{i}\right]}$
($i=1,\cdots,k+1$) with partition size smaller than $\varepsilon$
inductively, such that the resulting piecewise geodesic closed curve
over $[0,1]$ is simple. That will complete the proof of the theorem.

Let $\gamma^{\left(0\right)}$ be the Jordan curve such that 
\[
\gamma^{(0)}=\gamma,\ \mbox{on}\ \left[v_{0},u_{1}\right]\cup\left[v_{1},u_{2}\right]\cup\cdots\cup\left[v_{k-1},u_{k}\right]\cup\left[v_{k},u_{k+1}\right],
\]
and it is the minimizing geodesic (radial segment of the corresponding
normal ball) on each small interval of 
\[
\left[0,v_{0}\right],\left[u_{1},\tau_{1}\right],\left[\tau_{1},v_{1}\right],\cdots,\left[u_{k},\tau_{k}\right],\left[\tau_{k},v_{k}\right],\left[u_{k+1},1\right].
\]

By the construction in the proof of Theorem \ref{simple curve approximation},
we may find a partition 
\[
\mathcal{P}_{\left[v_{0},u_{1}\right]}^{\left(1\right)}:\ v_{0}=w_{0}^{\left(1\right)}<w_{1}^{\left(1\right)}<\cdots<w_{l_{1}-1}^{\left(1\right)}<w_{l_{1}}^{\left(1\right)}=u_{1}
\]
so that $\left\Vert \mathcal{P}_{[v_{0},u_{1}]}^{(1)}\right\Vert <\varepsilon$,
the geodesic interpolation $\gamma^{\mathcal{P}_{\left[v_{0},u_{1}\right]}^{\left(1\right)}}$
of $\gamma|_{\left[v_{0},u_{1}\right]}$ over the partition points
in $\mathcal{P}_{\left[v_{0},u_{1}\right]}^{\left(1\right)}$ is simple
and 
\begin{eqnarray*}
d\left(\gamma_{w_{i-1}^{(1)}},\gamma_{w_{i}^{(1)}}\right) & = & \delta{}_{\varepsilon}^{\left(1\right)},\ i=1,\cdots l_{1}-1,\\
d\left(\gamma_{w_{l_{1}-1}^{(1)}},\gamma_{u_{1}}\right) & \leqslant & \delta{}_{\varepsilon}^{\left(1\right)},
\end{eqnarray*}
for some $\delta_{\varepsilon}^{\left(1\right)}>0$. 

Moreover, we may choose $\delta_{\varepsilon}^{\left(1\right)}$small
enough so that dist$\left(\gamma{}^{\mathcal{P}_{\left[v_{0}u_{1}\right]}^{\left(1\right)}},\gamma^{\left(0\right)}|_{\left[\tau_{1},1\right]}\right)>0$
and $\delta_{\varepsilon}^{\left(1\right)}<\delta_{\varepsilon}$.

Now we will show that 
\[
\gamma^{\mathcal{P}_{\left[v_{0},u_{1}\right]}^{\left(1\right)}}\bigcap\gamma^{\left(0\right)}|_{\left[0,v_{0}\right)\cup\left(u_{1},\tau_{1}\right]}=\emptyset.
\]
In fact, if $\gamma^{\mathcal{P}_{\left[v_{0},u_{1}\right]}^{\left(1\right)}}\bigcap\gamma^{\left(0\right)}|_{\left[0,v_{0}\right)}\neq\emptyset,$
then from the construction of $\left\{ w_{i}^{\left(1\right)}\right\} $,
there exists some $i\geqslant2,$ such that $\gamma_{w_{i-1}^{(1)}},\gamma_{w_{i}^{(1)}}\in\overline{U_{k+1}}^{c}$
and 
\[
\overline{\gamma_{w_{i-1}^{\left(1\right)}}\gamma_{w_{i}^{\left(1\right)}}}\bigcap\gamma^{\left(0\right)}|_{\left[0,v_{0}\right)}\neq\emptyset,
\]
where $\overline{\gamma_{w_{i-1}^{\left(1\right)}}\gamma_{w_{i}^{\left(1\right)}}}$
denotes the image of the unique minimizing geodesic joining $\gamma_{w_{i-1}^{\left(1\right)}}$
and $\gamma_{w_{i}^{\left(1\right)}}$. However, since $d\left(\gamma_{w_{i-1}^{(1)}},\gamma_{w_{i}^{(1)}}\right)\leqslant\delta_{\varepsilon}^{\left(1\right)}<\delta_{\varepsilon}$,
we know from Lemma \ref{geodesic fact 2} that 
\[
d\left(\gamma_{v_{0}},\gamma_{w_{i-1}^{(1)}}\right)<\delta_{\varepsilon}^{\left(1\right)},\ d\left(\gamma_{v_{0}},\gamma_{w_{i}^{(1)}}\right)<\delta_{\varepsilon}^{\left(1\right)},
\]
which is an obvious contradiction to the construction of $\left\{ w_{i}^{\left(1\right)}\right\} {}_{i=0}^{l_{1}}$.
On the other hand, if $\gamma^{\mathcal{P}_{\left[v_{0},u_{1}\right]}^{\left(1\right)}}\bigcap\gamma^{\left(0\right)}|_{\left(u_{1},\tau_{1}\right]}\neq\emptyset,$
then there exists some $i\leqslant l_{1}-1,$ such that $\overline{\gamma_{w_{i-1}^{\left(1\right)}}\gamma_{w_{i}^{\left(1\right)}}}\bigcap\gamma^{\left(0\right)}|_{\left(u_{1},\tau_{1}\right]}\neq\emptyset.$

Since $\gamma_{w_{i-1}^{\left(1\right)}},\gamma_{w_{i}^{\left(1\right)}}\in\overline{U_{1}}^{c}$
and $d\left(\gamma_{w_{i-1}^{(1)}},\gamma_{w_{i}^{(1)}}\right)=\delta_{\varepsilon}^{\left(1\right)}<\delta_{\varepsilon}$,
we know again from Lemma \ref{geodesic fact 2} that 
\[
d\left(\gamma_{u_{1}},\gamma_{w_{i-1}^{(1)}}\right)<\delta_{\varepsilon}^{\left(1\right)},\ d\left(\gamma_{u_{1}},\gamma_{w_{i}^{(1)}}\right)<\delta_{\varepsilon}^{\left(1\right)}.
\]
But this is also a contradiction to the construction of $\left\{ w_{i}^{(1)}\right\} {}_{i=0}^{l_{1}}.$

Therefore, the closed curve $\gamma^{\left(1\right)}$ over $\left[0,1\right]$
defined by 
\[
\gamma_{t}^{\left(1\right)}=\begin{cases}
\gamma_{t}^{\mathcal{P}_{\left[v_{0},u_{1}\right]}^{\left(1\right)}}, & t\in\left[v_{0},u_{1}\right];\\
\gamma_{t}^{\left(0\right)}, & t\in\left[0,1\right]\backslash\left[v_{0},u_{1}\right],
\end{cases}
\]
is a Jordan curve.

Now consider $\gamma|_{\left[v_{1},u_{2}\right]}.$ The previous argument
can be carried through easily with respect to the Jordan curve $\gamma^{\left(1\right)}$,
and we obtain a finite partition 
\[
\mathcal{P}_{\left[v_{1},u_{2}\right]}^{\left(2\right)}:\ v_{1}=w_{0}^{\left(2\right)}<w_{1}^{\left(2\right)}<\cdots<w_{l_{2}-1}^{\left(2\right)}<w_{l_{2}}^{\left(2\right)}=u_{2},
\]
such that $\left\Vert \mathcal{P}_{\left[v_{1},u_{2}\right]}^{\left(2\right)}\right\Vert <\varepsilon,$
and the closed curve $\gamma^{\left(2\right)}$ over $\left[0,1\right]$
defined by
\[
\gamma_{t}^{\left(2\right)}=\begin{cases}
\gamma_{t}^{\mathcal{P}_{\left[v_{1},u_{2}\right]}^{\left(2\right)}}, & t\in\left[v_{1},u_{2}\right];\\
\gamma_{t}^{\left(1\right)}, & t\in\left[0,1\right]\backslash\left[v_{1},u_{2}\right],
\end{cases}
\]
is a Jordan curve, where $\gamma^{\mathcal{P}_{\left[v_{1},u_{2}\right]}^{\left(2\right)}}$
is the geodesic interpolation of $\gamma|_{\left[v_{1},u_{2}\right]}$
over the partition points in $\mathcal{P}_{\left[v_{1},u_{2}\right]}^{\left(2\right)}.$
By induction, we are able to construct simple piecewise geodesic approximation
of each piece of $\gamma$ outside $\cup_{i=1}^{k+1}\overline{U_{i}}$
and finally obtain a finite partition $\mathcal{P}_{\varepsilon}$
of $\left[0,1\right]$ with partition points 
\[
\left\{ 0\right\} \bigcup\left(\bigcup_{i=1}^{k+1}\left\{ v_{i-1},w_{1}^{\left(i\right)},\cdots,w_{l_{i}-1}^{\left(i\right)},u_{i},\tau_{i}\right\} ,\right)
\]
such that $\left\Vert \mathcal{P}_{\varepsilon}\right\Vert <\varepsilon,$
and the geodesic interpolation $\gamma^{\mathcal{P}_{\varepsilon}}$
(which is $\gamma^{\left(k+1\right)}$ by induction) of $\gamma$
over the points of $\mathcal{P}_{\varepsilon}$ is a Jordan curve. 

Now the proof is complete.\end{proof}

\begin{rem}
By slight modification of the proof, it is not hard to see that Theorem
\ref{Jordan curve approximation} also holds for  non-closed simple
curves. In this case, it strengthens the result of Theorem \ref{simple curve approximation}.
\end{rem}

\begin{rem}
It is possible to generalize our main results to infinite dimensional
spaces with suitable geodesic properties. For technical simplicity
we are not going to present the details.
\end{rem}

\section{Applications}

In this section, we shall demonstrate two applications of Theorem
\ref{Jordan curve approximation}. Here we assume that $M=\mathbb{R}^{2}.$

\subsection{Green's theorem for Jordan curves with finite $p$-variation ($1\leqslant p<2$)}

We will prove a generalized Green's theorem for planar Jordan curves
with finite $p$-variation, where $1\leqslant p<2$. First we shall
briefly recall basic facts about Young's integration. 
\begin{defn}
Let $\left(X,d\right)$ be a metric space. We say a function $\gamma:\left[0,1\right]\rightarrow X$
has finite $p$-variation if 
\[
\left\Vert \gamma\right\Vert _{p}:=\left(\sup_{\mathcal{P}}\sum_{t_{0}<t_{1}<\ldots<t_{n}}d\left(\gamma_{t_{i}},\gamma_{t_{i+1}}\right)^{p}\right)^{\frac{1}{p}}<\infty,
\]
where the supremum is taken over all finite partitions of $[0,1].$\end{defn}
\begin{thm}
(L.C. Young, \cite{Young}, see also \cite{St-flours}, Theorem 1.16)\label{Continuity theorem of Young integration}
Let $p,q\geqslant1$ be such that $\frac{1}{p}+\frac{1}{q}>1$. Let
$\gamma,\tilde{\gamma}:\left[0,1\right]\rightarrow\mathbb{R}^{d}$
be two continuous paths with finite $p$-variation and $q$-variation
respectively. Then the following limit exists: 
\[
\int_{0}^{1}\gamma\otimes\mathrm{d}\tilde{\gamma}:=\mbox{lim}_{\left|\mathcal{P}\right|\rightarrow0}\sum_{\mathcal{P}:\ 0=t_{0}<\ldots<t_{n}=1}\gamma_{t_{i}}\otimes\left(\tilde{\gamma}_{t_{i+1}}-\tilde{\gamma}_{t_{i}}\right).
\]
Moreover, the function $\int_{0}^{\cdot}\gamma\otimes\mathrm{d}\tilde{\gamma}$
has finite $q$-variation and 
\[
\left\Vert \int_{0}^{\cdot}\gamma\otimes\mathrm{d}\tilde{\gamma}\right\Vert _{q}\leqslant2\zeta\left(\frac{1}{p}+\frac{1}{q}\right)\left\Vert \gamma\right\Vert _{p}\left\Vert \tilde{\gamma}\right\Vert _{q},
\]
where $\zeta\left(\cdot\right)$ is the classical Riemann zeta function. 
\end{thm}
The following lemma demonstrates the importance of piecewise linear
approximation for the $p$-variation metric. 
\begin{lem}
\label{p variation norm of piecewise linear interpolation}Let $\gamma:\left[0,1\right]\rightarrow\mathbb{R}^{d}$
be a path with finite $p$-variation, where $p\geqslant1$. Let $f:\mathbb{R}^{d}\rightarrow\mathbb{R}^{\tilde{d}}$
be a Lipschitz function with Lipschitz constant $C$. Then 

1. (\cite{St-flours}, Lemma 1.12 and Lemma 1.18) $\left\Vert f\left(\gamma\right)\right\Vert _{p}\leqslant C\left\Vert \gamma\right\Vert _{p}$. 

2. (\cite{St-flours}, Proposition 1.14 and Remark 1.19) For all $q>p$,

\[
\left\Vert f\left(\gamma\right)-f\left(\gamma^{\mathcal{P}}\right)\right\Vert _{q}\rightarrow0
\]
as $\left\Vert \mathcal{P}\right\Vert \rightarrow0$. 
\end{lem}
We now prove Green's theorem for non-smooth Jordan curves. 
\begin{thm}
\label{Green's theorem for rough paths} Let $f,g:\mathbb{R}^{2}\rightarrow\mathbb{R}$
be functions with continuous first order derivatives, and let $\gamma:\left[0,1\right]\rightarrow\mathbb{R}^{2}$
be a positively oriented Jordan curve with finite $p$-variation,
where $1\leqslant p<2$. Let $x_{\cdot},y_{\cdot}$ denote the first
and second coordinate components of $\gamma_{\cdot}$ respectively.
Then 
\[
\int_{0}^{1}\left(f\left(\gamma_{s}\right)\mathrm{d}y_{s}-g\left(\gamma_{s}\right)\mathrm{d}x_{s}\right)=\int_{\mathrm{Int}(\gamma)}\left(\frac{\partial f}{\partial x}+\frac{\partial g}{\partial y}\right)\mathrm{d}x\mathrm{d}y,
\]
where the integral on the L.H.S. is understood as the Young's integral,
and $\mathrm{Int}\left(\gamma\right)$ denotes the interior of $\gamma.$ \end{thm}
\begin{proof}
Fix $\varepsilon>0$. According to Theorem \ref{Jordan curve approximation},
let $\mathcal{P}_{\varepsilon}$ be a finite partition of $\left[0,1\right]$
such that $\|\mathcal{P}_{\varepsilon}\|<\varepsilon,$ and the piecewise
linear interpolation $\gamma^{\mathcal{P}_{\varepsilon}}$ of $\gamma$
over the partition points in $\mathcal{P}_{\varepsilon}$ is a Jordan
curve. Let $x^{\mathcal{P}_{\varepsilon}},y^{\mathcal{P}_{\varepsilon}}$
be the first and second components of $\gamma^{\mathcal{P}_{\varepsilon}}$
respectively. It follows from the classical Green's theorem for piecewise
smooth Jordan curve that 
\[
\int_{0}^{1}(f\left(\gamma_{s}^{\mathcal{P}_{\varepsilon}}\right)\mathrm{d}y_{s}^{\mathcal{P}_{\varepsilon}}-g\left(\gamma_{s}^{\mathcal{P}_{\varepsilon}}\right)\mathrm{d}x_{s}^{\mathcal{P}_{\varepsilon}})=\int_{\mathrm{Int}\left(\gamma^{\mathcal{P}_{\varepsilon}}\right)}\left(\frac{\partial f}{\partial x}+\frac{\partial g}{\partial y}\right)\mathrm{d}x\mathrm{d}y.
\]

For any $q\in(p,2),$ we know that
\begin{align*}
 & \left|\int_{0}^{1}(f\left(\gamma_{s}^{\mathcal{P}_{\varepsilon}}\right)\mathrm{d}y_{s}^{\mathcal{P}_{\varepsilon}}-\int_{0}^{1}f\left(\gamma_{s}\right)\mathrm{d}y_{s})\right|\\
= & \left|\int_{0}^{1}\left(f\left(\gamma_{s}^{\mathcal{P}_{\varepsilon}}\right)-f\left(\gamma_{s}\right)\right)\mathrm{d}y_{s}^{\mathcal{P}_{\varepsilon}}+\int_{0}^{1}f\left(\gamma_{s}\right)\mathrm{d}\left(y_{s}^{\mathcal{P}_{\varepsilon}}-y_{s}\right)\right|\\
\leqslant & \left|\int_{0}^{1}\left(f\left(\gamma_{s}^{\mathcal{P}_{\varepsilon}}\right)-f\left(\gamma_{s}\right)\right)\mathrm{d}y_{s}^{\mathcal{P}_{\varepsilon}}\right|+\left|\int_{0}^{1}f\left(\gamma_{s}\right)\mathrm{d}\left(y_{s}^{\mathcal{P}_{\varepsilon}}-y_{s}\right)\right|\\
\leqslant & 2\zeta\left(\frac{2}{q}\right)\left(\left\Vert f\left(\gamma_{\cdot}^{\mathcal{P}_{\varepsilon}}\right)-f\left(\gamma_{\cdot}\right)\right\Vert _{q}\left\Vert \gamma\right\Vert _{q}+\left\Vert f\left(\gamma_{\cdot}\right)\right\Vert _{q}\left\Vert \gamma_{\cdot}^{\mathcal{P}_{\varepsilon}}-\gamma_{\cdot}\right\Vert _{q}\right),
\end{align*}
where the final inequality follows from Theorem \ref{Continuity theorem of Young integration}
and Lemma \ref{p variation norm of piecewise linear interpolation}.
Therefore, by Lemma \ref{p variation norm of piecewise linear interpolation},
\[
\int_{0}^{1}f\left(\gamma_{s}^{\mathcal{P}_{\varepsilon}}\right)\mathrm{d}y_{s}^{\mathcal{P}_{\varepsilon}}\rightarrow\int_{0}^{1}f\left(\gamma_{s}\right)\mathrm{d}y_{s}
\]
as $\varepsilon\rightarrow0$. Similarly, 
\[
\int_{0}^{1}g\left(\gamma_{s}^{\mathcal{P}_{\varepsilon}}\right)\mathrm{d}y_{s}^{\mathcal{P}_{\varepsilon}}\rightarrow\int_{0}^{1}g\left(\gamma_{s}\right)\mathrm{d}y_{s}
\]
as $\varepsilon\rightarrow0$. 

On the other hand, as $\gamma$ has finite $p$ variation, it has
a $\frac{1}{p}$-Hölder parametrization. Therefore, $\gamma$ has
Hausdorff dimension less than $2$. In particular, this means that
$\gamma([0,1])$ has zero Lebesgue measure. By applying the bounded
convergence theorem to the integrand 
\[
\left(\frac{\partial f}{\partial x}+\frac{\partial g}{\partial y}\right)\mathbf{1}_{\mathrm{Int}\left(\gamma^{\mathcal{P}_{\varepsilon}}\right)},
\]
we have 
\[
\int_{\mathrm{Int}\left(\gamma^{\mathcal{P}_{\varepsilon}}\right)}\left(\frac{\partial f}{\partial x}+\frac{\partial g}{\partial y}\right)\mathrm{d}x\mathrm{d}y\rightarrow\int_{\mathrm{Int}\left(\gamma\right)}\left(\frac{\partial f}{\partial x}+\frac{\partial g}{\partial y}\right)\mathrm{d}x\mathrm{d}y
\]
as $\varepsilon\rightarrow0$. 

Now the proof is complete.\end{proof}
\begin{rem}
For rectifiable paths, there is a version of Green's theorem which
works for non-simple closed curves, involving the winding number of
a path. An interesting inequality in this respect is the Banchoff-Pohl
inequality (see \cite{generalized iso.}), which generalizes the isoperimetric inequality and asserts
that the winding number of a rectifiable curve is square-integrable.
The reason for the ``simple closed'' condition in our version of
Green's theorem is because in general the winding number of a non-simple
non-rectifiable curve is not integrable. The fact that we can approximate
the rough Jordan curves by piecewise linear interpolations which are
still Jordan means that the winding number of each approximation is
an indicator function, which is bounded by the indicator function
of a neighborhood of Int$\left(\gamma\right)$.
\end{rem}

\begin{rem}
A direct consequence of Theorem \ref{Green's theorem for rough paths}
is Cauchy's theorem for Jordan curves with finite $p$-variation where
$1\leqslant p<2$, according to the Cauchy-Riemann equation for holomorphic
functions.
\end{rem}

\subsection{Uniqueness of signature for Jordan curves with finite $p$-variation
($1\leqslant p<2$)}

The sequence of iterated integrals (formally known as the signature)
of paths plays a key role in rough path theory. A central open problem
in this area is to determine a path from its iterated integrals. In
the case of bounded variation paths, Hambly and Lyons \cite{tree like}
proved that two paths of bounded variation can have the same sequence
of iterated integrals if and only if they can be obtained from each
other by a ''tree-like deformation''. In \cite{simple curve uniqueness },
it was proved that for planar simple curves with finite $p$-variation,
the sequence of iterated integrals of the path determines the path
up to reparametrization. In the context of stochastic processes, it was proved in
\cite{LQ12} that the sequence of iterated Stratonovich's integrals
of Brownian motion determines the Brownian sample paths almost surely.
This result was extended to diffusion processes in \cite{Geng-inversion}. 

Here we are going to prove the uniqueness of signature for planar
Jordan curves with finite $p$-variation, where $1\leqslant p<2$
is fixed throughout the rest of this section.

We shall follow \cite{MR1654527} and embed the sequence of iterated
integrals of a path into the tensor algebra, which gives us a very
nice algebraic structure to work with. 

Assume that $\left(\mathbb{R}^{d}\right)^{\otimes n}$ is the tensor
produce space equipped with the Euclidean norm by identifying it with
$\mathbb{R}^{d^{n}}$. Let 
\[
T\left(\mathbb{R}^{d}\right)=\oplus_{n=0}^{\infty}\left(\mathbb{R}^{d}\right)^{\otimes n},
\]
and let $\pi_{n}:T\left(\mathbb{R}^{d}\right)\rightarrow\left(\mathbb{R}^{d}\right)^{\otimes n}$
denote the projection map. Assume that $\left\{ \mathbf{e}_{1},\ldots,\mathbf{e}_{d}\right\} $
is the standard basis of $\mathbb{R}^{d}$, and $\left\{ \mathbf{e}_{1}^{*},\ldots,\mathbf{e}_{d}^{*}\right\} $
is the corresponding dual basis of $\mathbb{R}^{d*}$. We embed $T\left(\left(\mathbb{R}^{d}\right)^{*}\right)$
into $T\left(\mathbb{R}^{d}\right)^{*}$ by extending the relation
\begin{eqnarray*}
\mathbf{e}_{i_{1}}^{*}\otimes\ldots\otimes\mathbf{e}_{i_{n}}^{*}\left(\mathbf{e}_{j_{1}}\otimes\ldots\otimes\mathbf{e}_{j_{k}}\right) & =\begin{cases}
1 & \mbox{if \ensuremath{n=k}\ and \ensuremath{i_{1}=j_{1},\cdots,i_{k}=j_{k}},}\\
0 & \mbox{otherwise,}
\end{cases}
\end{eqnarray*}
linearly. 
\begin{defn}
Let $\gamma:\left[0,1\right]\rightarrow\mathbb{R}^{d}$ be a continuous
path with finite $p$-variation, then the formal series of tensors
\[
S\left(\gamma\right)_{0,1}:=1+\sum_{i=1}^{\infty}\int_{0<s_{1}<\ldots<s_{i}<1}\mathrm{d}\gamma_{s_{1}}\otimes\ldots\otimes\mathrm{d}\gamma_{s_{i}},
\]
defined in terms of Young's integrals, is called the signature of
$\gamma$ over $\left[0,1\right]$. 
\end{defn}
We will briefly recall three important properties of signature.
\begin{prop}
\label{Properties of signature} Let $\gamma:\left[0,1\right]\rightarrow\mathbb{R}^{d}$
be a continuous path with finite $p$-variation. Then

1. (\cite{St-flours}, p. 32) Let $r:\left[0,1\right]\rightarrow\left[0,1\right]$
be a continuous increasing function, then 
\[
S\left(\gamma_{\cdot}\right)_{0,1}=S\left(\gamma_{r\left(\cdot\right)}\right)_{0,1}.
\]

2. (\cite{St-flours}, p. 32) For all $a\in\mathbb{\mathbb{R}}^{d}$,
\[
S\left(a+\gamma_{\cdot}\right)_{0,1}=S\left(\gamma_{\cdot}\right)_{0,1}.
\]

3. (\cite{St-flours}, Corollary 2.11) Let $\gamma_{n}$ be a sequence
of paths with finite $p$-variation and 
\[
\left\Vert \gamma-\gamma_{n}\right\Vert _{p}\rightarrow0,
\]
as $n\rightarrow\infty$, then for each $k\in\mathbb{N}$, 
\[
\left|\pi_{k}\left(S\left(\gamma_{n}\right)_{0,1}\right)-\pi_{k}\left(S\left(\gamma\right)\right)\right|\rightarrow0
\]
as $n\rightarrow\infty.$
\end{prop}

It turns out that some terms in the signature of a curve can be reduced
to a single line integral. This is the key idea to prove our uniqueness
of signature result. 
\begin{prop}
\label{Fubini and Green theorem} Let $\gamma$ be a positively oriented
Jordan curve with finite $p$-variation. Let $x_{\cdot},y_{\cdot}$
be the first and second coordinate components of $\gamma$ respectively.
Then for any $k,n\geqslant 0$
\begin{eqnarray*}
 &  & \mathbf{e}_{1}^{*\otimes (k+1)}\otimes\mathbf{e}_{2}^{*\otimes (n+1)}\left(S\left(\gamma\right)_{0,1}\right)\\
 & = & \int_{0}^{1}\int_{0}^{s_{n+k+2}}\ldots\int_{0}^{s_{2}}\mathrm{d}x_{s_{1}}\ldots\mathrm{d}x_{s_{k+1}}\mathrm{d}y_{s_{k+2}}\ldots\mathrm{d}y_{s_{n+k+2}}\\
 & = & \frac{1}{k!n!}\int_{\mathrm{Int}(\gamma)}\left(x-x_{0}\right)^{k}\left(y_{1}-y\right)^{n}\mathrm{d}x\mathrm{d}y.
\end{eqnarray*}
\end{prop}
\begin{proof}
Note that 
\begin{eqnarray*}
 &  & \int_{0}^{1}\int_{0}^{s_{n+k+2}}\ldots\int_{0}^{s_{2}}\mathrm{d}x_{s_{1}}^{\mathcal{P}_{\varepsilon}}\ldots\mathrm{d}x_{s_{k+1}}^{\mathcal{P}_{\varepsilon}}\mathrm{d}y_{s_{k+2}}^{\mathcal{P}_{\varepsilon}}\ldots\mathrm{d}y_{s_{n+k+2}}^{\mathcal{P}_{\varepsilon}}\\
 & = & \frac{1}{\left(k+1\right)!n!}\int_{0}^{1}\left(x_{s_{k+1}}^{\mathcal{P}_{\varepsilon}}-x_{0}^{\mathcal{P}_{\varepsilon}}\right)^{k+1}\left(y_{1}^{\mathcal{P}_{\varepsilon}}-y_{s_{k+1}}^{\mathcal{P}_{\varepsilon}}\right)^{n}\mathrm{d}y_{s_{k+1}}^{\mathcal{P}_{\varepsilon}}\\
 & = & \frac{1}{k!n!}\int_{\mathrm{Int}\left(\gamma^{\mathcal{P}_{\varepsilon}}\right)}\left(x-x_{0}\right)^{k}\left(y_{1}-y\right)^{n}\mathrm{d}x\mathrm{d}y.
\end{eqnarray*}
By Lemma \ref{Properties of signature}, 
\begin{eqnarray*}
 &  & \int_{0}^{1}\int_{0}^{s_{n+k+2}}\ldots\int_{0}^{s_{2}}\mathrm{d}x_{s_{1}}^{\mathcal{P}_{\varepsilon}}\ldots\mathrm{d}x_{s_{k+1}}^{\mathcal{P}_{\varepsilon}}\mathrm{d}y_{s_{k+2}}^{\mathcal{P}_{\varepsilon}}\ldots\mathrm{d}y_{s_{n+k+2}}^{\mathcal{P}_{\varepsilon}}\\
 & \rightarrow & \int_{0}^{1}\int_{0}^{s_{n+k+2}}\ldots\int_{0}^{s_{2}}\mathrm{d}x_{s_{1}}\ldots\mathrm{d}x_{s_{k+1}}\mathrm{d}y_{s_{k+2}}\ldots\mathrm{d}y_{s_{n+k+2}}
\end{eqnarray*}
as $\varepsilon\rightarrow0$. 

As in the proof of Theorem \ref{Green's theorem for rough paths},
\[
\int_{\mathrm{Int}\left(\gamma^{\mathcal{P}_{\varepsilon}}\right)}\left(x-x_{0}\right)^{k}\left(y_{1}-y\right)^{n}\mathrm{d}x\mathrm{d}y\rightarrow\int_{\mathrm{Int}\left(\gamma\right)}\left(x-x_{0}\right)^{k}\left(y_{1}-y\right)^{n}\mathrm{d}x\mathrm{d}y
\]
as $\varepsilon\rightarrow0$. Therefore, the result follows. 
\end{proof}

\begin{rem}
The case of $k=1,n=0$ for Proposition \ref{Fubini and Green theorem}
has been proved by Werness \cite{Wer12}. The main difficulty in extending
to the general case involves mainly the interchange of iterated integrals. 
\end{rem}
The following lemma is the main reason why our result only works for
Jordan curves. 
\begin{lem}
\label{Jordan curve image determine the curve}Let $\gamma,\widetilde{\gamma}:\left[0,1\right]\rightarrow\mathbb{R}^{2}$
be two positively oriented Jordan curves such that $\gamma\left(\left[0,1\right]\right)=\widetilde{\gamma}\left(\left[0,1\right]\right)$
and $\gamma_{0}=\widetilde{\gamma}_{0}$. There exists a continuous
increasing function $r:\left[0,1\right]\rightarrow\left[0,1\right]$
such that $\gamma_{r\left(t\right)}=\widetilde{\gamma}$. In other
words, $\gamma$ and $\widetilde{\gamma}$ are equal up to a reparametrization. \end{lem}
\begin{proof}
As $\gamma$ is a Jordan curve,$\gamma([0,1])\backslash\gamma_{0}$
and $\widetilde{\gamma}([0,1])\backslash\widetilde{\gamma}_{0}$ are
both homeomorphic to $\left(0,1\right)$. Therefore, the function
$r:\left(0,1\right)\rightarrow\left(0,1\right)$ defined by $r\left(t\right)=\gamma^{-1}\circ\widetilde{\gamma}\left(t\right)$
is a homeomorphism $\left(0,1\right)\rightarrow\left(0,1\right)$.
Hence, it is strictly monotone. This implies that $\lim_{t\rightarrow0}r\left(t\right)$
exists. Moreover, it is easy to see that $\lim_{t\rightarrow0}r\left(t\right)\in\left\{ 0,1\right\} $. 

If $\lim_{t\rightarrow0}r\left(t\right)=0$, then $r$ can be extended
to a continuous increasing function on $\left[0,1\right]$. As $\gamma_{r\left(t\right)}=\widetilde{\gamma}_{t}$,
we know that $\gamma$ and $\widetilde{\gamma}$ equal up to reparametrization.
If $\lim_{t\rightarrow0}r\left(t\right)=1,$ then $\lim_{t\rightarrow1}r\left(t\right)=0$
and $r\left(t\right)$ is decreasing. This implies that $r\left(1-t\right)$
is an increasing continuous function. Therefore, $\gamma$ and $\widetilde{\gamma}$
have opposite orientations, which is a contradiction.
\end{proof}

Now we are in position to state and prove our result on the uniqueness
of signature for planar Jordan curves.
\begin{thm}
\label{Uniqueness of signature for Jordan curve} Let $\gamma,\tilde{\gamma}:\left[0,1\right]\rightarrow\mathbb{R}^{2}$
be a Jordan curves with finite $p$-variation. Then $S\left(\gamma\right)_{0,1}=S\left(\tilde{\gamma}\right)_{0,1}$
if and only if $\gamma$ and $\tilde{\gamma}$ is a translation and
a reparametrization of each other. \end{thm}
\begin{proof}
Sufficiency follows from Proposition \ref{Properties of signature}.
We now consider the necessity part.

By applying a translation we may assume that $\gamma_{1}=\gamma_{0}=\tilde{\gamma}_{0}=\tilde{\gamma}_{1}=0$. 

As $\mathbf{e}_{1}^{*}\otimes\mathbf{e}_{2}^{*}\left(S\left(\gamma\right)_{0,1}\right)=\mathbf{e}_{1}^{*}\otimes\mathbf{e}_{2}^{*}\left(S\left(\tilde{\gamma}\right)_{0,1}\right)$,
by Proposition \ref{Fubini and Green theorem} we have 
\[
\left(-1\right)^{\varepsilon\left(\gamma\right)}\int_{\mathrm{Int}\left(\gamma\right)}\mathrm{d}x\mathrm{d}y=\left(-1\right)^{\varepsilon\left(\tilde{\gamma}\right)}\int_{\mathrm{Int}\left(\tilde{\gamma}\right)}\mathrm{d}x\mathrm{d}y,
\]
where $\varepsilon\left(\gamma\right)$ is $0$ if $\gamma$ is positively
oriented and $1$ otherwise. As $\int_{\mathrm{Int}\left(\gamma\right)}\mathrm{d}x\mathrm{d}y$
and $\int_{\mathrm{Int}\left(\tilde{\gamma}\right)}\mathrm{d}x\mathrm{d}y$
are both positive, we must have $\gamma$ and $\tilde{\gamma}$ oriented
in the same direction. 

Without loss of generality, assume both $\gamma$ and $\tilde{\gamma}$
are positively oriented. By Proposition \ref{Fubini and Green theorem}
and that $S\left(\gamma\right)_{0,1}=S\left(\tilde{\gamma}\right)_{0,1}$,
we have 
\[
\int_{\mathrm{Int}\left(\gamma\right)}\left(x-x_{0}\right)^{k}\left(y_{1}-y\right)^{n}\mathrm{d}x\mathrm{d}y=\int_{\mathrm{Int}\left(\tilde{\gamma}\right)}\left(x-\tilde{x}_{0}\right)^{k}\left(\tilde{y}_{1}-y\right)^{n}\mathrm{d}x\mathrm{d}y
\]
for all $k,n\geq0$. Therefore, 
\[
\int_{\mathrm{Int}\left(\gamma\right)}e^{i(\lambda_{1}x+\lambda_{2}y)}\mathrm{d}x\mathrm{d}y=\int_{\mathrm{Int}\left(\tilde{\gamma}\right)}e^{i(\lambda_{1}x+\lambda_{2}y)}\mathrm{d}x\mathrm{d}y
\]
for all $\lambda_{1},\lambda_{2}\in\mathbb{R}$. 

Both $\mathbf{1}_{\mathrm{Int}\left(\gamma\right)}$ and $\mathbf{1}_{\mathrm{Int}\left(\tilde{\gamma}\right)}$
are in $L^{1}$ and by the injectivity of the Fourier transform on
$L^{1}$, we have 
\[
\mathbf{1}_{\mathrm{Int}\left(\gamma\right)}\left(x,y\right)=\mathbf{1}_{\mathrm{Int}\left(\tilde{\gamma}\right)}\left(x,y\right)
\]
for almost every $\left(x,y\right)\in\mathbb{R}^{2}$. In particular,
this implies that both $\mathrm{Int}\left(\gamma\right)\backslash\overline{\mathrm{Int}\left(\tilde{\gamma}\right)}\subset\mbox{\ensuremath{\mathrm{Int}}}\left(\gamma\right)\backslash\mathrm{Int}\left(\tilde{\gamma}\right)$
and $\mathrm{Int}\left(\tilde{\gamma}\right)\backslash\overline{\mathrm{Int}\left(\gamma\right)}\subset\mathrm{Int}\left(\gamma\right)\backslash\mathrm{Int}\left(\tilde{\gamma}\right)$
are null sets in $\mathbb{R}^{2}$. However, since both $\mathrm{Int}\left(\tilde{\gamma}\right)\backslash\overline{\mathrm{Int}\left(\gamma\right)}$
and $\mathrm{Int}\left(\gamma\right)\backslash\overline{\mathrm{Int}\left(\tilde{\gamma}\right)}$
are open, they must be empty. Therefore, 
\[
\overline{\mathrm{Int}\left(\tilde{\gamma}\right)}=\overline{\mathrm{Int}\left(\gamma\right)}.
\]
By the Jordan curve theorem, we have 
\[
\overline{\mathbb{R}^{2}\backslash\overline{\mathrm{Int}\left(\tilde{\gamma}\right)}}=\mathbb{R}^{2}\backslash\mathrm{Int}\left(\tilde{\gamma}\right).
\]
Therefore, $\mathrm{Int}\left(\tilde{\gamma}\right)=\mathrm{Int}\left(\gamma\right)$
and so $\gamma([0,1])=\tilde{\gamma}([0,1])$ and by Lemma \ref{Jordan curve image determine the curve},
$\gamma$ and $\tilde{\gamma}$ are equal up to reparametrization. 
\end{proof}

\begin{rem}
The proof of Theorem \ref{Uniqueness of signature for Jordan curve}
gives a very explicit way of computing the moments of the finite measure
$\mathbf{1}_{\mathrm{Int}\left(\gamma\right)}\left(x,y\right)\mathrm{d}x\mathrm{d}y$
from the signature of $\gamma.$ However, due to possible difficulty
of numerically inverting the Fourier transform, this is still not
an explicit reconstruction scheme for a Jordan curve from its signature.
In fact, an explicit reconstruction scheme in general remains a significant
open problem in rough path theory. 
\end{rem}
\section*{Acknowledgement}
The authors wish to thank Professor Terry Lyons for his valuable suggestions on the present paper. The authors are supported by the Oxford-Man Institute at University of Oxford. The first author is also supported by ERC (Grant Agreement No.291244 Esig).

\end{document}